\documentclass[11pt]{amsart}

\def\FINAL{1}
\def\WEBTITLES{1}
\advance\textheight by 6mm
\advance\headheight by -2mm

\usepackage{skew}
\usepackage{tkz-graph}
\usetikzlibrary{calc}

\let\ring\circ
\def\Cbar{\bar\C}
\def\Kbar{\bar K}
\def\kbar{{\bar k}}
\def\Abar{\bar\A}
\def\H{H^1_{\text{\'et}}}
\def\II{{\mathcal I}}
\def\JJ{{\mathcal J}}
\def\KK{{\mathcal K}}

\begin{document}

\title{Tate module and bad reduction}
\author{Tim and Vladimir Dokchitser,  Adam Morgan}

\address{Department of Mathematics, University of Bristol, Bristol BS8 1TW, UK}
\email{tim.dokchitser@bristol.ac.uk}
\address{Department of Mathematics, University College London, Gower Street, London, WC1E 6BT, UK}
\email{v.dokchitser@ucl.ac.uk}
\address{Max-Planck-Institut für Mathematik, Vivatsgasse 7, 53111 Bonn,
Germany}
\email{am516@mpim-bonn.mpg.de}


\begin{abstract}
Let $C/K$ be a curve over a local field. 
We study the natural semilinear action of Galois on the minimal regular model of $C$ over a field $F$ where it becomes semistable.
This allows us to describe the Galois action on the $l$-adic Tate module of the Jacobian of $C/K$ in terms of the special fibre of this model over $F$. 
\end{abstract}

\llap{.\hbox to 20em{\hfill}}
\vskip -1cm

\maketitle

\section{Introduction}

Let $C/K$ be a curve\footnote{Smooth, proper, geometrically connected.} of positive genus over a non-Archimedean local field\footnote{Our convention is that a local field is a discretely valued field with finite residue field.}, with Jacobian $A/K$.
Our goal is to describe the action of the absolute Galois group $G_K$ on the $l$-adic Tate module $T_l A$ 
in terms of the reduction of $C$ over a field where $C$ becomes semistable,
for $l$ different from the residue characteristic.

Fix a finite Galois extension $F/K$ over which $C$ is semistable \cite{DM}.
Write $\mathcal{O}_F$ for the ring of integers of $F$, $k_F$ for the residue field of $F$, $I_F$ for the inertia group, 
$\cC/\mathcal{O}_F$ for the minimal regular model of $C/F$, and $\cC_{k_F}/k_F$ for its
special fibre. For any field $L$, we denote by $\bar{L}$ the separable closure of $L$.

Grothendieck defined a canonical filtration 
by $G_F$-stable $\Z_l$-lattices \cite[IX, \S12]{SGA7I},
\begin{equation}\label{eq1}
  0\subset T_l(A)^t \subset T_l(A)^{I_F} \subset T_l(A);
\end{equation}
 $T_l(A)^t$ is sometimes referred to as the ``toric part''.
He showed that the graded pieces of the filtration are unramified $G_F$-modules and are, canonically,
\begin{equation}\label{eq2}
  H^1(\Upsilon,\Z) \tensor_\Z\Z_l(1), \qquad T_l \Pic^0 \tilde\cC_{\bar k_F}, \qquad  H_1(\Upsilon,\Z) \tensor_\Z\Z_l,
\end{equation}%
where $\tilde\cC_{\bar k_F}$ is the normalisation of $\cC_{\bar k_F}$, 
$\Upsilon$ is the dual graph of $\cC_{{\bar k}_F}$ 
(a vertex for each irreducible component and an edge for every ordinary double point) 
and $H^1, H_1$ are singular (co)homology groups. Here the middle piece may be further decomposed as\footnote{Here
$\Ind_H^G(\cdot)$ stands for ${\Z_l[G]}\tensor_{\Z_l[H]}(\cdot)$.} 
\begin{equation}\label{eqind}
  T_l \Pic^0 (\tilde \cC_{\bar k_F}) \iso \bigoplus_{\Gamma\in \JJ/G_F} \Ind_{\Stab(\Gamma)}^{G_F} T_l\Pic^0(\Gamma),
\end{equation}
where $\JJ$ is the set of connected components of $\tilde\cC_{\bar k_F}$. 

In particular (cf. \cite[\S2.10]{CFKS}), the above discussion determines the first $l$-adic \'etale cohomology group of $C$ as a $G_F$-module:
\begin{equation}\label{eq3}
  \H(C_{\Kbar},\Q_l) \>\>\iso\>\> H^1(\Upsilon,\Z)\!\tensor\!\Sp_2 \>\oplus\> \H(\tilde\cC_{\bar k_F},\Q_l),
\end{equation}
where $\Sp_2$ is the  2-dimensional `special' representation  (see \cite[4.1.4]{Ta}).

In this paper we describe the full $G_K$-action on $T_l(A)$ in terms of this filtration, even though $C$ may not be semistable over $K$. 


\begin{theorem}
\label{main}
The filtration \eqref{eq1} of $T_l(A)$ is independent of the choice of $F/K$ and is $G_K$-stable. 
Moreover, $G_K$ acts semilinearly\footnote{see Definition \ref{def:semilinearaction}} 
on $\mathcal{C}/\mathcal{O}_F$, inducing actions on $\cC_{k_F}$, $\Upsilon$, $\Pic^0 \cC_{\bar k_F}$ 
and $\Pic^0 \tilde\cC_{\bar k_F}$, with respect to which  \eqref{eq2} identifies the graded pieces 
as $G_K$-modules and \eqref{eqind} extends to a $G_K$-isomorphism
$$
  T_l \Pic^0 (\tilde \cC_{\bar k_F}) \iso \bigoplus_{\Gamma\in \JJ/G_K} \Ind_{\Stab(\Gamma)}^{G_K} T_l\Pic^0(\Gamma).
$$
The action of $\sigma \in G_K$ on $\cC_{k_F}$ is uniquely determined by its action on non-singular points, where  it  is given by 
$$
  \qquad
  \cC_{k_F}({\bar k}_F)_{\ns} \!\overarrow{\text{lift}}\! \cC(\mathcal{O}_{F^{\textup{nr}}})=C(F^{\textup{nr}}) 
    \!\overarrow{\sigma}\! C(F^{\textup{nr}}) =\cC(\mathcal{O}_{F^{\textup{nr}}}) \!\overarrow{\text{reduce}}\! \cC_{k_F}({\bar k}_F)_{\ns},
$$
where $F^{\textup{nr}}$ denotes the maximal unramified extension of $F$.
(Whilst there may be many choices of lift, the composite map from left to right is independent of this choice, cf. Theorem \ref{geommain} (3).)
\end{theorem}
%
%

\begin{corollary} 
\label{cor16}
There is an isomorphism of $G_K$-modules
$$
\begin{array}{llllll}
  \H(C_{\Kbar},\Q_l) &\iso& H^1(\Upsilon,\Z)\!\tensor\!\Sp_2 \>\oplus\> \H(\tilde\cC_{{\bar k}_F},\Q_l) \\[5pt]
       &\iso& \displaystyle H^1(\Upsilon,\Z)\!\tensor\!\Sp_2 \>\oplus\> 
      \bigoplus_{\Gamma\in \JJ/G_K} \Ind_{\Stab(\Gamma)}^{G_K} \H(\Gamma,\Q_l).
\end{array}
$$
\end{corollary}

\begin{remark}
Suppose $\sigma\in \Stab_{G_K}(\Gamma)$ acts on $\bar{k}_F$ as a non-negative integer power of Frobenius $x\mapsto x^{|k_K|}$. Its (semilinear) action on the points of $\Gamma(\bar{k}_F)$ coincides with the action of a $k_F$-linear morphism (see Remark \ref{rmk-frob}). In particular, one can determine trace of $\sigma$ on $\H(\Gamma,\Q_l)=\Hom(T_l\Pic^0(\Gamma),\Q_l)$ using the Lefschetz trace formula and counting fixed points of this morphism on $\Gamma(\bar{k}_F)$. See \cite[\S6]{hq} for an explicit example.

\end{remark}

\begin{remark}
For the background in the semistable case see
\cite[\S12.1-\S12.3, \S12.8]{SGA7I} when $k=\bar k$ and \cite[\S9.2]{BLR} 
or \cite{Pap} in general.
In the non-semistable case, the fact that the inertia group of $F/K$ acts on $A$ by geometric
automorphisms goes back to Serre--Tate \cite[Proof of Thm. 2]{ST},
and \cite[pp. 12--13]{CFKS} explains how to extend this to a semilinear action 
of the whole of $G_K$. We also note that in \cite[Thm. 2.1]{BW} the $I_K$-invariants of $T_l A$ ($A$ a Jacobian) are described  in terms of the quotient curve by the Serre--Tate action. 
\end{remark}

We now illustrate how one might use Theorem \ref{main} in two simple examples.

\begin{example}
\label{exa1}
Let $p>3$ be a prime. Fix a primitive 3rd root of unity $\zeta\in \bar \Q_p$ and $\pi=\sqrt[3]p$, and let $F=\Q_p(\zeta,\pi)$.
Consider the elliptic curve 
$$
  E/\Q_p \colon y^2=x^3+p^2 
$$ 
which has additive reduction 
over $\Q_p$. Over $F$, the substitution $u= \frac{x}{\pi^2}, z= \frac{y}{p}$ results in the equation 
\begin{equation} \label{curve example equation}
z^2=u^3+1,
\end{equation}
so that $E$ attains good reduction over $F$ with the special fibre of its minimal model the curve $\bar{E}/k_F$ given by (reducing modulo $p$) equation (\ref{curve example equation}).

The Galois group $G_{\Q_p}$ acts on $\bar{E}$ by semilinear morphisms, which by Theorem \ref{main} are given on $\bar{E}(\bar\F_p)$ by the ``lift-act-reduce'' procedure. Explicitly, we compute the action of $\sigma \in G_{\Q_p}$ on a 
 point $(u_0,z_0)\in\bar{E}(\bar\F_p)$, with lift $(\tilde u_0,\tilde z_0)$ to the model of $E$ with good reduction. On the original equation for $E$ this corresponds to the point $(\pi^2  \tilde u_0, p \tilde z_0)\in E(F)$. Acting on this point by $\sigma$, rewriting the result in the variables $u,z$, and then reducing to $\bar{E}$ results in the point $(\bar\zeta^{2\chi(\sigma)}\bar\sigma u_0,\bar\sigma z_0)\in \bar{E}(\bar\F_p)$ where $\bar\sigma$ is the induced action of $\sigma$ on the residue field and $\chi$ is defined by $\frac{\sigma(\pi)}{\pi}\equiv \bar\zeta^{\chi(\sigma)}\mod \pi$. In summary, the ``lift-act-reduce'' procedure is given by
$$\scalebox{0.9}{$
 (u_0,z_0) \!\to\! (\tilde u_0, \tilde z_0) \!\to\!  (\pi^2  \tilde u_0, p \tilde z_0) \!\to\! (\sigma (\pi^2\tilde u_0), p \sigma \tilde z_0) \!\to\!
(\tfrac{\sigma\pi^2}{\pi^2} \sigma \tilde u_0, \sigma\tilde z_0) \!\to\! (\bar\zeta^{2\chi(\sigma)}\bar\sigma u_0,\bar\sigma z_0).$}
$$
 In particular, $\sigma$ in the inertia group of $\mathbb{Q}_p$ acts as the geometric automorphism $(u,z)\mapsto (\bar\zeta^{2\chi(\sigma)}u,z)$ of $\bar{E}$.

By Theorem \ref{main}, $T_l(E)$ with the usual Galois action is isomorphic to $T_l(\bar{E})$ with the action induced by the semilinear automorphisms. In particular, we see that the action factors through $\Gal(F^{nr}/\Q_p)$. Moreover the inertia subgroup acts by elements of order 3
(as expected from the N\'eron--Ogg--Shafarevich criterion), and the usual actions of $G_{\Q_p(\pi)}$ on $T_l(E)$ and $T_l(\bar{E})$ agree under the reduction map.
\end{example}

\begin{example}
\label{exa2}
 As in Example \ref{exa1} let $p>3$, $\zeta\in\bar\Q_p$ a primitive 3rd root of unity, $\pi=\sqrt[3]p$, and $\chi$ defined by $\frac{\sigma(\pi)}{\pi}\equiv \bar\zeta^{\chi(\sigma)}\mod \pi$. 
Consider the hyperelliptic curve
$$
  C/\Q_p \colon y^2=\bigl((x\!-\!\pi)^2-p\bigr)\bigl((x\!-\!\zeta\pi)^2-p\bigr)\bigl((x\!-\!\zeta^2\pi)^2-p\bigr)
$$
(note that $C$ is indeed defined over $\mathbb{Q}_p$ since any element of $G_{\mathbb{Q}_p}$ just permutes the factors on the right hand side of the defining equation).
Over $F=\Q_p(\zeta,\pi)$ the substitution $x'= \frac{x}{\pi}, y'= \frac{y}{p}$ transforms it to
$$
  y^2=\bigl((x\!-\!1)^2-\pi\bigr)\bigl((x\!-\!\zeta)^2-\pi\bigr)\bigl((x\!-\!\zeta^2)^2-\pi\bigr).
$$
This is a semistable curve that reduces to 
$$
  \bar C\colon y^2 = (x-1)^2(x-\bar\zeta)^2(x-\bar\zeta^2)^2, 
$$
a union of rational curves $y=x^3-1$ and $-y=x^3-1$ meeting at 3 points $(1,0)$, $(\bar\zeta,0)$ and $(\bar\zeta^2,0)$.
The dual graph $\Upsilon$ of $\bar C$ is 

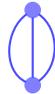
\begin{center}
\def\LoopW(#1){
  \path[draw,-,thick,color=blue!70] (#1) edge[out=155,in=90] ($(#1)-(1.3,0)$);
  \path[draw,-,thick,color=blue!70] (#1) edge[out=210,in=270] ($(#1)-(1.3,0)$);
}
\begin{tikzpicture}[scale=1]
  \SetVertexNormal[Shape=circle, FillColor=blue!50, LineColor=blue!50, LineWidth=0.8pt]
    \tikzset{VertexStyle/.append style = {inner sep=0.5pt,minimum size=0.5em,font = \tiny\bfseries}}
  \Vertex[x=1.50,y=0.000,L=\relax]{1};
  \Vertex[x=1.50,y=1.000,L=\relax]{2};
  \SetUpEdge[lw=0.8pt,color=blue!70]
  \path[draw,-,thick,color=blue!70] (1) edge[out=155,in=215] (2);
  \path[draw,-,thick,color=blue!70] (1) edge[out=35,in=-35] (2);
  \path[draw,-,thick,color=blue!70] (1) edge (2);
\end{tikzpicture}
\end{center}

\noindent
We compute analogously to Example \ref{exa1} that $\sigma\in G_{\Q_p}$ acts as the semilinear automorphism 
of $\bar{C}$
$$
 \phantom{hahahahahahaha} (x,y)\mapsto (\bar\zeta^{\chi(\sigma)}\bar \sigma  x, \bar \sigma y).
$$
On $\Upsilon$ the action fixes the vertices (the two components of $\bar C$), and permutes the edges 
through a natural action of $G=\Gal(F/\Q_p)\iso S_3$ when $p\equiv 2$ mod 3, and 
of $G=\Gal(F/\Q_p)\iso C_3$ when $p\equiv 1$ mod 3.
Thus $H_1(\Upsilon,\Z)$ is the sum-zero part of $\Z[S_3/C_2]$, respectively $\Z[C_3]$, as a $\Z G$-module.
By Theorem \ref{main}, the Tate module $T_l(\Jac C)$ is an extension of 
$H_1(\Upsilon,\Z) \tensor_\Z\Z_l$ by $H^1(\Upsilon,\Z) \tensor_\Z\Z_l(1)$.
Now choose a topological generator $\sigma$ of the tame inertia and a Frobenius element $\phi$ of $G_{\Q_p(\pi)}$.
There is a $\Q_l$-basis of the special representation $\Sp_2$ on which they act as
$$
  \sigma \mapsto \smallmatrix 1101, \qquad \phi \mapsto \smallmatrix 100p,
$$
see \cite[4.1.4]{Ta}. After tensoring with $H^1(\Upsilon,\Z)$, 
Corollary \ref{cor16} describes $\H(C_{\Kbar},\Q_l)$ explicitly as follows:
in some basis, $\sigma$ acts as $\Sigma$, and $\phi$ acts as $\Phi_1$ when $p\equiv 1\mod 3$ and
acts as $\Phi_2$ when $p\equiv 2\mod 3$, where
$$
\smaller[3]
\Phi_1=
\begin{pmatrix}
1&0&0&0\cr
0&1&0&0\cr
0&0&p&0\cr
0&0&0&p\cr
\end{pmatrix}\!,
\quad\>\>
\Phi_2=
\begin{pmatrix}
-1&1&0&0\cr
0&1&0&0\cr
0&0&-p&p\cr
0&0&0&p\cr
\end{pmatrix}\!,
\quad\>\>
\Sigma=
\begin{pmatrix}
0&-1&0&-1\cr
1&-1&1&-1\cr
0&0&0&-1\cr
0&0&1&-1\cr
\end{pmatrix}\!.
$$
\end{example}

\subsection*{Layout}
To prove Theorem \ref{main}, we review semilinear actions in 
\S\ref{semilinear actions section}, and prove a general theorem (\ref{geommain}) 
for models of schemes that are sufficiently `canonical' to admit a unique extension of 
automorphisms of the generic fibre; in particular, this applies to minimal regular models and stable models
of curves, and N\'eron models of abelian varieties (this again goes back to 
\cite[Proof of Thm. 2]{ST}). We then apply this result in \S\ref{curves and jacobians sect}
to obtain Theorem \ref{main}.

In fact, all our results 
are slightly more general, 
and apply to $K$ the fraction field of an arbitrary Henselian DVR with perfect residue field, and
not just for the Galois action but also the action of other (e.g. geometric) automorphisms.

For applications of the results of the paper to the arithmetic of curves we refer the reader to \cite[\S6]{hq} and \cite[\S10]{M2D2}.
In particular, for hyperelliptic curves $y^2=f(x)$ over local fields of odd residue characteristic, 
\cite{M2D2} describes the Galois representation of the curve in terms of the arithmetic 
of the roots of~$f$.

\subsection*{Acknowledgements}
We would like to thank the University of Warwick and Baskerville Hall 
where parts of this research were carried out. 
We would like to thank the anonymous referees for their thorough reading of the paper 
and several comments improving the exposition.
This research is supported by \hbox{EPSRC} grants EP/M016838/1 and EP/M016846/1 `Arithmetic of hyperelliptic curves'.
The second author is supported by a Royal Society University Research Fellowship.

\section{Semilinear actions} \label{semilinear actions section}

\begin{notation}
For schemes $X/S$ and $S'/S$ we denote by $X_{S'}/S'$ the base change $X\times_S S'$.  
 For a scheme $T/S$ we write  $X(T)=\textup{Hom}_S(T,X)$ for the $T$-points of $X$. For a ring $R$, by an abuse of notation we write $X(R)=X(\textup{Spec }R)$. For morphisms $X\stackrel{f}{\longrightarrow}Y\stackrel{g}{\longrightarrow}Z$ our convention is that the composition is denoted $g\circ f$.
\end{notation}

\subsection{Semilinear morphisms}

\begin{definition}
If $S$ is a scheme, $\alpha\in\Aut S$, and $X$ and $Y$ are $S$-schemes, 
a morphism $f: X\to Y$ is \emph{$\alpha$-linear} (or simply \textit{semilinear}) if the following diagram commutes:
$$
\begin{CD}
  X        @>f>>              Y       \\
    @VVV                              @VVV        \\
  S @>\alpha>>     S  \\  
\end{CD}
$$
\end{definition}

\begin{definition} \label{twisted scheme}
For a scheme $X/S$ and an automorphism $\alpha\in\Aut(S)$, write $X_\alpha$
for $X$ viewed as an $S$-scheme via $X\to S\overarrow{\alpha} S$.
\end{definition}

\begin{remark} \label{remtwist}
 An $\alpha$-linear morphism $X\to X$ is the same as an $S$-morphism $X_\alpha\to X$.
Note further that
\begin{itemize}
\item 
$X_{\alpha \beta}=(X_\beta)_\alpha$,
\item
an $S$-morphism $f: X\to X$ induces an $S$-morphism $\alpha(f): X_\alpha\to X_\alpha$, 
which is the same map as $f$ on the underlying schemes.
\end{itemize}
\end{remark}

\begin{remark} \label{rembasechange}
Equivalently,  $X_\alpha=X\times_{S,\alpha^{-1}}S$ viewed as an $S$-scheme via the second projection, where the notation indicates that we are using the morphism $\alpha^{-1}:S\rightarrow S$ to form the fibre product. More precisely, the first projection gives an isomorphism of $S$-schemes $X\times_{S,\alpha^{-1}}S\rightarrow X_\alpha$. 
\end{remark}

\begin{lemma}
\label{lemproj}
Let $X$, $Y$, $S'$ be $S$-schemes, $\alpha\in\Aut S$ and suppose we are given an $\alpha$-linear morphism
$f: X\to Y$ and an $\alpha$-linear automorphism $\alpha': S'\to S'$.
\begin{enumerate}
\item There is a unique $\alpha'$-linear
morphism $f\times_\alpha \alpha': X_{S'}\to Y_{S'}$ such that 
$\pi_Y\ring (f\times_\alpha \alpha') = f\ring \pi_X$, where $\pi_X$ and $\pi_Y$ are the projections $X_{S'}\rightarrow X$ and $Y_{S'}\rightarrow Y$ respectively. 
\item Given another $S$-scheme $Z$, $\beta\in \Aut S$, $g: Y\to Z$ a $\beta$-linear morphism and $\beta': S'\to S'$ a $\beta$-linear automorphism, we have 
$$
(g\times_{\beta} {\beta'}) \ring   (f\times_{\alpha} {\alpha'})= (g\ring f)\times_{\beta\ring \alpha} ({\beta'}\ring {\alpha'}).
$$
\end{enumerate}
\end{lemma}

\begin{proof}
(1) By the universal property of the fibre product $Y_{S'}$ applied to 
 the morphisms $f\circ\pi_X: X_{S'}\to Y$ and $\alpha'\circ\pi_{S'}: X_{S'}\to S'$ 
there is a unique morphism $X_{S'}\to Y_{S'}$ with the required properties.

(2) Follows from the uniqueness of the morphisms afforded by (1).
\end{proof}

%
%
%

\subsection{Semilinear actions}

\begin{notation}
For a group $G$ acting on a scheme $X$, for each $\sigma\in G$ we write $\sigma_X$ (or just $\sigma$) for the associated automorphism of $X$. All actions considered are left actions.
\end{notation}

\begin{definition} \label{def:semilinearaction}
Let $G$ be a group and  $S$ a scheme on which $G$ acts.  We say that $G$ acts \emph{semilinearly} on an  $S$-scheme $X/S$ if $G$ acts on $X$ as a scheme, and if for each $\sigma \in G$ the automorphism $\sigma_X$ is $\sigma_S$-linear. 
\end{definition}

\begin{remark}
\label{cocycle}
Specifying a semilinear action of $G$ on $X/S$ is equivalent
to giving $S$-isomorphisms $c_\sigma: X_{\sigma}\to X$ for each $\sigma\in G$,
satisfying the cocycle condition
$
  c_{\sigma \tau}=c_\sigma \ring \sigma(c_\tau)
$
(cf. Remark \ref{remtwist}).
\end{remark}

\begin{definition}[Action on points] \label{defpointact}
Given a semilinear action of $G$ on $X/S$ and $T/S$, 
$G$ acts on $X(T)$ via
$$
  P \quad\longmapsto\quad \sigma_X\ring P\ring \sigma_T^{-1}.
$$
\end{definition}

\begin{definition}[Base change action]
Suppose $G$ acts semilinearly on $X/S$. Then given $S'/S$ and a semilinear action of $G$ on $S'$,  we get a semilinear
\emph{base change action} of $G$ on $X_{S'}/S'$ by setting, for $\sigma \in G$, 
$$
\sigma_{X_{S'}}=\sigma_X\times_{\sigma_S}\sigma_{S'}.
$$
\end{definition}

%

\begin{lemma}
\label{lemactpoints}
Suppose $G$ acts semilinearly on $X/S$ and $T/S$. 
\begin{enumerate}
\item 
If $G$ acts semilinearly on $Y/S$ and $f:X\rightarrow Y$ is $G$-equivariant, then so is the natural map $X(T)\rightarrow Y(T)$ given by $P\mapsto f\circ P$. 
\item 
If $G$ acts semilinearly on $T'/S$ and $f:T'\rightarrow T$ is $G$-equivariant, 
then so is the natural map $X(T)\to X(T')$ given by $P\mapsto P\circ f$.
\item 
If $G$ acts semilinearly on $S'/S$ then the natural map $X(T)\to X_{S'}(T_{S'})$ given by
$P\mapsto P\times_{\id}\id$
is equivariant for the action of $G$, where $G$ acts on $X_{S'}(T_{S'})$ via
base change.
\end{enumerate}
\end{lemma}

\begin{proof}
(1) Clear. (2). Clear. (3) 
Denoting by $\phi$ the map $X(T)\to X_{S'}(T_{S'})$ in the statement, for  each $\sigma\in G$ we have by Lemma \ref{lemproj} (2) that 
$$
  \sigma\cdot \phi(P) = (\sigma_X\times_{\sigma_S}\sigma_{S'}) \ring (P\times_{\id}\id) \ring (\sigma_{T}\times_{\sigma_S}\sigma_{S'})^{-1} =
  (\sigma_X\ring P\ring \sigma_T^{-1}) \times_{\id} \id = \phi(\sigma\cdot P)
$$
as desired.
\end{proof}

\begin{example}[Automorphisms]
Let $X$ be an $S$-scheme and $G=\Aut_S X$. Then the natural action of $G$ on $X$ is semilinear 
for the trivial action on~$S$. 
Given $T/S$ with trivial $G$-action, the induced action of $\sigma \in G$ on 
$X(T)$  recovers the usual action  $ P \mapsto \sigma\ring P$.
\end{example}

\begin{example}[Galois action] \label{galois example1}
Let $K$ be a field, $G=G_K$ and $S=\Spec K$ with trivial $G$ action. Let $T=\Spec\Kbar$ with $\sigma\in G$ acting via  
$$
(\sigma^{-1})^*: \Spec \Kbar\to \Spec \Kbar.
$$
Then for any scheme $X/K$, letting $G$ act trivially on $X$, the action on $X(\bar{K})$ is $
  P \mapsto P\ring \sigma^*
$,
which is just the usual Galois action on points. 

Now let $F/K$ be Galois, so that the $G$-action on $\textup{Spec }\bar{K}$ restricts to an action on  $\textup{Spec }F$. We obtain an example of a genuinely semilinear action  by considering the base change action of $G$ on $X_F$, so that here the action on the base $\textup{Spec }F$ is through $(\sigma^{-1})^*$. The natural map $X(\bar K)\to X_F(\bar K)$ is an equality, and identifies the $G$-actions by Lemma \ref{lemactpoints} (3).
\end{example}
\section{Geometric action over local fields} \label{geom act section}

Let $\mathcal{O}$ be a Henselian DVR, $K$ its field of fractions, $F/K$ a finite Galois extension, $\mathcal{O}_F$ the integral closure of $\mathcal{O}$ in $F$, and $k_F$ the residue field of $\mathcal{O}_F$. Let $G$ be a group equipped with a homomorphism  $\theta:G\rightarrow G_K$ (in our applications we will either take $G=G_K$ (and $\theta$ the identity map), or $\theta$ the zero-map).  This induces an action of $G$ on $\textup{Spec }\bar{K}$  via $\sigma \mapsto (\theta(\sigma)^{-1})^*$, which restricts to actions on $\textup{Spec }F$, $\textup{Spec }\mathcal{O}_F$, etc. 

Now let $X/F$ be a scheme on which $G$ acts semilinearly
with respect to the above action on $\Spec F$. Denote by $\mathcal{O}_F^{\textup{sh}}$ the strict Henselisation of $\mathcal{O}_F$, and $F^{\textup{sh}}$ the fraction field of  $\mathcal{O}_F^{\textup{sh}}$, noting that the map $\theta$ induces actions of $G$ on $\mathcal{O}_F^{\textup{sh}}$ and $F^{\textup{sh}}$.



\def\FB{F^{\nr}}  
\def\OFB{O_{F^{\nr}}}

\def\FB{\bar F} 
\def\OFB{\mathcal{O}_{\bar F}}

\begin{theorem}
\label{geommain}
Suppose $\cX/\mathcal{O}_F$ is a model\footnote{For our purposes, a \emph{model} $\cX/\mathcal{O}_F$ of $X$ is simply a scheme over $\mathcal{O}_F$ with a specified isomorphism
$i: \cX\times_{\mathcal{O}_F}F\overarrow{\iso} X$.} of $X$ such that for each $\sigma\in G$ the semilinear morphism $\sigma_{X}$ extends uniquely  to a semilinear morphism $\sigma_{\cX}: \cX\to \cX$. Then
\begin{enumerate}
\item 
The map $\sigma \mapsto \sigma_{\cX}$ defines a semilinear action of $G$ on $\cX/\mathcal{O}_F$. In particular, it induces by base-change a semilinear action of $G$ on the special fibre $\cX_{k_F}$, and also induces actions on $\cX(\mathcal{O}_F^{\textup{sh}})$ and $\cX_{k_F}(\bar k_F)$.
\item
The natural maps on points $\cX(\mathcal{O}_F^{\textup{sh}})\to X(F^\textup{sh})$ and $\cX(\mathcal{O}_F^{\textup{sh}})\to \cX_{k_F}(\bar k_F)$
are $G$-equivariant.
\item
Suppose $\cX(\mathcal{O}_F^{\textup{sh}})\to X(F^\textup{sh})$ is bijective,
and let $I$ be the image of $\cX(\mathcal{O}_F^{\textup{sh}})\to \cX_{k_F}(\bar k_F)$.
Then the action of  $\sigma\in G$ on $I$  is given~by
$$
  I \overarrow{\text{lift}} \cX(\mathcal{O}_F^{\textup{sh}})
    \overarrow{=} X(F^\textup{sh}) \overarrow{\sigma} X(F^\textup{sh}) \overarrow{=}
    \cX(\mathcal{O}_F^{\textup{sh}}) \overarrow{\text{reduce}} I.
$$
\end{enumerate}
\end{theorem}

\begin{proof}
(1) Follows from uniqueness of the extension of the $\sigma_{X_F}$ to $\cX$.

 (2)
Follows from Lemma \ref{lemactpoints} (3) applied to the natural maps \[\mathcal{O}_F^{\textup{sh}}\tensor_{\mathcal{O}_F} F\rightarrow F^\textup{sh}~~\textup{ and }~~\OFB\tensor_{\mathcal{O}_F} k_F\to k_{\FB}.\]

(3) Follows from (2).
\end{proof}

\begin{remark}
The assumption on the uniqueness of the extensions of the $\sigma_{X}$ is automatic if $\cX/\mathcal{O}_F$ is separated.
The assumption that $\cX(\mathcal{O}_F^\textup{sh})\to X(F^\textup{sh})$ is bijective in part (3) is automatic if 
$\cX/\mathcal{O}_F$ is proper. 
\end{remark}



\begin{remark}\label{rmk-frob}
Suppose $\vchar k_F=p>0$  and $\sigma\in G$ acts on $\bar k_F$  as $x\mapsto x^{p^n}$ for some $n\geq 0$. Let $\Fr$ denote the $p^n$-power absolute Frobenius.
Note that $\Fr:\cX_{k_F}\to \cX_{k_F}$ is $\Fr=\sigma_{\textup{Spec }k_F}^{-1}$-linear whilst $\sigma_{\cX_{k_F}}$ is $\sigma_{\textup{Spec }k_F}$-linear, so that $\psi_\sigma=\sigma_{\cX_{k_F}}\circ \Fr$ is a 
$k_F$-morphism. Moreover,  since absolute Frobenius commutes with all scheme morphisms, for any $P\in \cX_{k_F}(\bar k_F)$  we have 
$$
\psi_\sigma(P) = 
\sigma_{\cX_{k_F}}\circ \Fr \circ P =  
\sigma_{\cX_{k_F}} \circ P \circ \Fr =  
\sigma_{\cX_{k_F}} \circ P \circ \sigma_{\textup{Spec }k_F}^{-1} = \sigma \cdot P.
$$
In particular,  the action of $\sigma$ on the $\bar k_F$-points of $\cX_{k_F}$ agrees with that of a $k_F$-morphism, even though the action of $\sigma$ on $k_F$ may be non-trivial.  
\end{remark}

\begin{remark} \label{assumptionshold} 
The assumptions of Theorem
\ref{geommain}, including (3), hold in the following situations:
\begin{itemize}
\item[(i)]
$X/F$ a curve of positive genus and $\cX/\mathcal{O}_F$ the minimal proper regular model of $X/F$.
\item[(ii)] 
$X/F$ a curve of positive genus, and $\cX/\mathcal{O}_F$ the stable model of $X/F$ (provided $X$ is semistable over $F$).
\item[(iii)]
$X/F$ an  abelian variety and $\cX/\mathcal{O}_F$ the N\'eron model of $X/F$.
\item[(iv)] $X/F$ a curve of positive genus  and $\cX/\mathcal{O}_F$ the N\'eron model of $X/F$ in the sense of \cite{LT}.
\end{itemize}
To see that the assumption of the theorem is satisfied, use Remark \ref{cocycle}:
in all three
cases, for any $\sigma \in G$, $\cX_\sigma$ is again a model of $X_\sigma$ of the same type as $\cX$, 
and the universal properties that these models satisfy guarantee the existence and 
uniqueness of the extensions. Regarding (3),
$\cX(\mathcal{O}_F^\textup{sh})=\cX_F(F^\textup{sh})$ by properness in (i),(ii) and the N\'eron mapping property in (iii) and (iv).
Moreover, we note that the image $I$ of the reduction map contains all non-singular points since $\mathcal{O}_F^\textup{sh}$ is Henselian. Note that in particular if we take $G=G_K$ (and $\theta$ the identity map), $C/K$ a curve of positive genus, and  $X=C_F$ equipped with the canonical semilinear action of $G_K$ as described in Example \ref{galois example1}, then   the assumptions of Theorem \ref{geommain} are satisfied for $\cX/\mathcal{O}_F$ any one of the minimal proper regular model, the stable model (assuming $C$ becomes semistable over $F$) or  N\'eron model of $X$. Similarly we may take $A/K$ to be an abelian variety, $X$ the base change of $A$ to $F$ equipped with its canonical semilinear action, and $\cX/\mathcal{O}_F$ the N\'eron model of $X/F$.
\end{remark}

\section{Curves and Jacobians} \label{curves and jacobians sect}

As in \S\ref{geom act section}, let $\mathcal{O}$ be a Henselian DVR, $K$ its field of fractions, $F/K$ a finite Galois extension, $\mathcal{O}_F$ the integral closure of $\mathcal{O}$ in $F$ and $k_F$ the residue field of $\mathcal{O}_F$. From now on we assume that $k_F$ is perfect. Denote by $\mathcal{O}_F^{\textup{sh}}$ the strict Henselisation of $\mathcal{O}_F$, and $F^{\textup{sh}}$ the fraction field of  $\mathcal{O}_F^{\textup{sh}}$. Let $G$ be a group equipped with a homomorphism  $\theta:G\rightarrow G_K$, acting on  $\textup{Spec }\bar{K}$  via $\sigma \mapsto (\theta(\sigma)^{-1})^*$, and hence also on $\textup{Spec }F$, $\textup{Spec }\mathcal{O}_F$, etc. Finally, fix a  curve $C/F$ of positive genus and semistable reduction equipped with a semilinear action of $G$ (with respect to the above action on $\textup{Spec} F$) and let $A/F$ be the Jacobian of $C$.  For the application to Theorem \ref{main} we take $G=G_K$ (and $\theta$ the identity map), begin with a curve over $K$ which becomes semistable over $F$, and take $C$ to be the base change of this curve to $F$ along with the canonical semilinear action of $G_K$ as in Example \ref{galois example1} and Remark \ref{assumptionshold}. 

Let $\cC/\mathcal{O}_F$ be the minimal regular model of $C/F$ (which is semistable since $C/F$ is). 
Let $\A/\mathcal{O}_F$ be the N\'eron model of $A/F$ with special fibre $\Abar/k_F$, and let
$\A^0/\mathcal{O}_F$ be its identity component with special fibre~$\Abar^0/k_F$.   Theorem \ref{geommain} and Remark \ref{assumptionshold} then provide a semilinear action of $G$ on $\cC/\mathcal{O}_F$, inducing a semilinear action on  the special fibre $\cC_{k_F}/k_F$ also.  Next,  let $ \Pic^0_{\cC/\mathcal{O}_F} $ denote the identity component of the relative Picard functor of $\cC$ over $\mathcal{O}_F$. This inherits a semilinear action $\sigma \mapsto  (\sigma_{\cC}^{-1})^*$ of $G$ induced from that on $\mathcal{C}/\mathcal{O}_F$ as we now explain (since pull back of line bundles is contravariant one needs to include the inverse to obtain a left action). By Remark \ref{rembasechange} and the fact that the relative Picard functor commutes with base change, it also commutes with twisting in the sense of Definition \ref{twisted scheme}: for all $\sigma \in G$ we have $
  \Pic^0_{\cC_\sigma/\mathcal{O}_F} =  (\Pic^0_{\cC/\mathcal{O}_F})_\sigma.
$
Functoriality of $\Pic^0_{\cC/\mathcal{O}_F}$ combined with Remark \ref{remtwist} gives the sought automorphism $(\sigma_{\cC}^{-1})^*$.  We note that this induces by base-change a semilinear action of $G$ on the special fibre $\Pic^0_{\cC_{k_F}/k_F}$, with $\sigma\in G$ acting as $(\sigma_{\cC_{k_F}}^{-1})^*$.  Further, the argument above with $C/F$ in place of $\mathcal{C}/\mathcal{O}_F$ yields a semilinear action of $G$ on the Jacobian $A/F$, again given by $\sigma \mapsto (\sigma_C^{-1})^*$ (if we take $G=G_K$ and $C$ arising via base change from $K$ then this is the usual Galois  action on the Jacobian on $C$). Now Theorem \ref{geommain} and Remark \ref{assumptionshold} apply once again to give a semilinear action of $G$ on  $\A/\mathcal{O}_F$ which induces semilinear actions on $\A^0/\mathcal{O}_F$, $\Abar/k_F$, and $\Abar^0/k_F$ also. We will need the following compatibility result between the above actions.
%

\begin{lemma} \label{compatibility of geom actions}
For any $\sigma \in G$, the following diagram commutes
$$
\begin{CD}
  \A^0            @>\sigma_{\A}>>       \A^0       \\ 
  @VV\iso V                      @VV\iso V   \\  
  \Pic^0_{\cC/\mathcal{O}_F}  @>(\sigma_{\cC}^{-1})^*>>      \Pic^0_{\cC/\mathcal{O}_F},   \\  
\end{CD}
$$
with the vertical isomorphisms provided by \cite[Thm. 9.5.4]{BLR}.
\end{lemma}

\begin{proof}
Since $\A^0$ and $\Pic^0_{\cC/\mathcal{O}_F}$ are separated over $\mathcal{O}_F$ it suffices to check that the diagram commutes on the generic fibre, where it does by the definition of the action of $G$ on $A$.
\end{proof}

We now turn to the $G$-action on the Tate module of $A$. Here we write $[m]$ for the set of 
$m$-torsion points over the separable closure of the ground field.

\begin{lemma} \label{limit of ST cor}
\textup{(1)}
For every $m\ge 1$ coprime to $\vchar k_F$, 
\[A[m]^{I_F}\iso \Abar[m]\] as $G$-modules, where here $I_F:=\textup{Gal}(\bar{F}/F^\textup{sh})$ is the inertia group of $F$. 

\textup{(2)}
For every $l\ne\vchar k_F$, $$
    T_l A^{I_F} \iso T_l \Abar = T_l \Abar^0 \iso T_l \Pic^0_{\cC_{k_F}/k_F}
$$
as $G$-modules.
\end{lemma}

\begin{proof}
(1) Note that $A[m]^{I_F}=A(F^{\textup{sh}})[m]$ is a $G$-submodule of $A[m]$ since $G$ acts on  $F^{\textup{sh}}$. By \cite[Lemma 2]{ST}, under the reduction map $A[m]^{I_F}$ is isomorphic to $\Abar[m]$ as abelian groups, and this map is $G$-equivariant for the given actions by
Theorem \ref{geommain} (2). 

(2) Pass to the limit in (1) and apply Lemma \ref{compatibility of geom actions} for the final isomorphism.
\end{proof}

The following theorem describes the $G$-module $T_l  \Pic^0_{\cC_{k_F}/k_F}$. We begin by explaining how $G$ acts on certain objects associated to $\cC_{k_F}$. 

\begin{remark} \label{rem:graphact}
Let $Y=\cC_{\bar k_F}$. Combining the action of $G$ on $\mathcal{C}_{k_F}$ with the action on $\bar{k}_F$ coming from the homomorphism $\theta:G\rightarrow G_K$ we obtain by base-change a semilinear action of $G$ on $Y$. This moreover induces a semilinear action on the  normalisation $\tilde Y$ of $Y$ (any automorphism of $Y$, semilinear or otherwise, lifts uniquely to $\tilde{Y}$ and the lifts of the $\sigma_Y$ are easily checked to define a semilinear action of $G$). Write 
\begingroup
\def\Cbar{Y}
\def\Cbarn{{\tilde Y}}
\def\Ish{{\mathbb I}}
\def\nmap{n}
\def\Yb{Y_{\kbar}}
\def\Ybn{\tilde Y_{\kbar}}

\begin{tabular}{llllll}
$\nmap$      &=& normalisation map $\Cbarn\to\Cbar$,\cr
$\II$        &=& set of singular (ordinary double) points of $\Cbar$,\cr
$\JJ$        &=& set of connected components of $\Cbarn$,\cr
$\KK$        &=& $\nmap^{-1}(\II)$; this comes with two canonical maps\cr
&& $\phi: \KK\to \II$, $P\mapsto\nmap(P),$\cr
&& $\psi: \KK\to \JJ$, $P\mapsto$ component of $\Cbarn$ on which $P$ lies.\cr
\end{tabular}

\noindent The dual graph $\Upsilon$ of $Y$ has vertex set $\JJ$ and edge set $\II$. $\KK$ is the set of edge endpoints, and the maps $\phi$ and $\psi$ specify adjacency (note that loops and multiple edges are allowed).  A graph automorphism of $\Upsilon$ (which we allow to permute multiple edges and swap edge endpoints) is precisely
the data of bijections $\KK\to\KK$, $\II\to\II$ and $\JJ\to\JJ$ that commute with $\phi$ and $\psi$. In this way,  the action of $G$ on $\tilde{Y}$ induces an action of $G$ on $\Upsilon$, and hence also on $H_1(\Upsilon,\Z)$ and $H^1(\Upsilon,\Z).$
\endgroup
\end{remark}

\begin{theorem} \label{dual graph theorem}
We have an exact sequence of $G$-modules
$$
  0 \lar H^1(\Upsilon,\Z) \tensor_\Z\Z_l(1) \lar T_l  \Pic^0_{\cC_{k_F}/k_F} \lar T_l  \Pic^0_{\tilde \cC_{k_F}/k_F} \lar 0
$$
where $\Upsilon$ is the dual graph of $\cC_{\bar k_F}$ and $\tilde \cC_{k_F}$ the normalisation of $\cC_{k_F}$.
Moreover,
$$
T_l  \Pic^0_{\tilde \cC_{k_F}/k_F}\iso \bigoplus_{\Gamma\in \JJ/G} \Ind_{\Stab(\Gamma)}^{G} T_l\Pic^0(\Gamma)
$$
where $\JJ$ is the set of geometric connected components of $\tilde\cC_{\bar k_F}$. 
(The action of $G$ on $\Z_l(1)$ is via the map $\theta:G\rightarrow G_K$.)
\end{theorem}

\begin{proof}
\begingroup
\def\Cbar{Y}
\def\Cbarn{{\tilde Y}}
\def\Ish{{\mathbb I}}
\def\nmap{n}
\def\Yb{Y_{\kbar}}
\def\Ybn{\tilde Y_{\kbar}}

We follow \cite[pp. 469--474]{SGA7I} closely, except our sequences (\ref{PicSeq}) and (\ref{mayer}) are slightly tweaked from the ones appearing there, and we must check $G$-equivariance of all maps appearing. Write $k=k_F$, $Y=\cC_{\bar k_F}$ and let $\tilde{Y}$, $n$, $\mathcal{I}$, $\mathcal{J}$, $\mathcal{K}$,  $\phi$, $\psi$ be as in Remark \ref{rem:graphact}. 
%
%
%
The normalisation map $\nmap$ is an isomorphism outside $\II$,
and yields an exact sequence of sheaves on $\Cbar$
$$
  1 \lar \mathcal{O}_{\Cbar}^\times \lar \nmap_* \mathcal{O}_{\Cbarn}^\times \lar \Ish \lar 0,
$$
with $\Ish$ concentrated in $\II$. Consider the long exact sequence on cohomology
$$
  0 \to H^0(\Cbar,\mathcal{O}_{\Cbar}^\times) \to H^0(\Cbarn,\mathcal{O}_{\Cbarn}^\times) \to H^0(\Cbar,\Ish) \to 
    H^1(\Cbar,\mathcal{O}_{\Cbar}^\times) \to H^1(\Cbarn,\mathcal{O}_{\Cbarn}^\times) \to 0
$$
which is surjective on the right since $\Ish$ is flasque.
Writing $(\kbar^\times)^\II$ for the set of functions $\II\to\kbar^\times$, and similarly
for $\JJ$ and $\KK$, we have
$$
  H^0(\Cbar,\Ish) = \coker((\kbar^\times)^\II\overarrow{\phi^*}(\kbar^\times)^\KK),
$$
where $\phi^*$ takes a function $\II\to \kbar^\times$ to $\KK\to \kbar^\times$ by 
composing with $\phi$. With $\psi^*$ defined in the same way, 
the exact sequence above becomes
\begin{equation}\label{PicSeq}
  0 \lar \kbar^\times \lar (\kbar^\times)^\JJ \overarrow{\psi^*} 
    \frac{(\kbar^\times)^\KK}{\phi^*((\kbar^\times)^\II)} \lar \Pic \Cbar(\kbar) 
    \lar \Pic \Cbarn (\kbar) \lar 0. 
\end{equation}
Write the dual graph $\Upsilon$ as the union
$\Upsilon=U\cup V$, where $U$ 
is the union of open edges, and $V$ is the union of small open neighbourhoods of
the vertices. Then the Mayer-Vietoris sequence reads
\begin{equation}\label{mayer} 
  0 \lar H_1(\Upsilon,\Z) \lar \Z^\KK \overarrow{(\phi,\psi)} \Z^\II\times\Z^\JJ \lar \Z \lar 0,
\end{equation}
since $H_0(U)=\Z^\II$, $H_0(V)=\Z^\JJ$, $H_0(U\cap V)=\Z^\KK$ 
and the higher homology groups of $U$, $V$ and $U\cap V$ all vanish. 
%

Now take $\sigma\in G$. Since the semilinear action of $G$ on $\tilde{Y}$ lifts that on $Y$, the natural maps $\mathcal{O}_\Cbar\to (\sigma_Y)_* \mathcal{O}_\Cbar$ and $\mathcal{O}_\Cbarn\to (\sigma_{\tilde{Y}})_* \mathcal{O}_\Cbarn$
 give the left two vertical maps in the commutative diagram 
$$
\begin{CD}
  0  @>>>     \mathcal{O}_\Cbar^\times    @>>>   n_* \mathcal{O}_\Cbarn^\times   @>>>   \Ish  @>>>   0   \\
    @.        @VVV         @VVV        @VVV           \\
  0  @>>>     (\sigma_Y)_* \mathcal{O}_\Cbar^\times       @>>>   (\sigma_Y)_* n_* \mathcal{O}_\Cbarn^\times      @>>>   (\sigma_Y)_*\Ish     @>>>   0,  \\
\end{CD}
$$
 with these two vertical maps then giving rise to the third.
Taking the long exact sequences for cohomology associated to this diagram 
 we find  that  \eqref{PicSeq} is an exact sequence of 
$G$-modules  (note that as $\sigma_Y$ is an isomorphism,  for any sheaf $\cF$ on $Y$ the natural pullback map on cohomology identifies $H^i(Y,(\sigma_Y)_*\cF)$ with $H^i(Y,\cF)$ for all $i$).

On the level of Tate modules $T_l$ ($l\ne\vchar k$), \eqref{PicSeq} then yields an exact sequence of $G$-modules\footnote{For $l\ne\vchar k$ the first three terms of (\ref{PicSeq}) are $l$-divisible, from which it follows that the sequence is also exact on the level of $l^n$-torsion for each $n\geq 1$. Moreover, since for all $n$ the $l^n$-torsion in each term is a finite abelian group, the resulting inverse systems all satisfy the Mittag-Leffler conditions. In particular, the sequence of $l$-adic Tate modules is exact also.}
$$
  0 \!\lar\! \Z_l(1) \!\lar\! \Z_l[\II](1)\oplus \Z_l[\JJ](1) \!\lar\! \Z_l[\KK](1) \!\lar\! 
  T_l \Pic \Cbar \!\lar\! T_l \Pic \Cbarn \!\lar\! 0
$$
with $G$ acting on $\Z_l(1)$ via the map $\theta:G\rightarrow G_K$ and on $\mathcal{I}, \mathcal{J}$ and $\mathcal{K}$ by permutation.
On the other hand, applying $\Hom(-,\Z_l(1))$ to \eqref{mayer} yields an exact sequence of $G$-modules
$$
  0 \lar \Z_l(1) \lar \Z_l[\II](1)\oplus \Z_l[\JJ](1) \lar \Z_l[\KK](1) 
    \lar H^1(\Upsilon,\Z)\tensor_\Z\Z_l(1) \lar 0.
$$
The first claim follows.

For the second claim, note that $T_l\Pic^0\Cbarn=\bigoplus_{\Gamma\in\JJ}T_l\Pic^0\Gamma$ abstractly, and that once the $G$-action is accounted for the right hand side becomes the asserted direct sum of induced modules.
\endgroup
\end{proof}

\begin{remark} \label{rem:toric part}
Under the Serre--Tate isomorphism $T_l  \Pic^0_{\cC_{k_F}/k_F} \iso T_l(A)^{I_F}$, the subspace $H^1(\Upsilon,\Z) \tensor_\Z\Z_l(1)$ maps onto $T_l(A)^t$.
 To see this, let $\cF$ be the image of $H^1(\Upsilon,\Z) \tensor_\Z\Z_l(1)$ in $T_l(A)$. In the notation of Theorem \ref{dual graph theorem}, the quotient of $T_l(A)^{I_F}$ by $\mathcal{F}$ is isomorphic to $\bigoplus_{\Gamma\in \JJ/G} \Ind_{\Stab(\Gamma)}^{G} T_l\Pic^0(\Gamma)$ and as such is free as a $\mathbb{Z}_l$-module. In particular $\mathcal{F}$ is a saturated submodule of $T_l(A)^{I_F}$. Similarly, $T_l(A)^t$ is also a saturated submodule of $T_l(A)^{I_F}$ by \eqref{eq1}. Since also $\mathcal{F}$ and $T_l(A)^t$ have the same $\Z_l$-rank (again by \eqref{eq1} and Theorem \ref{dual graph theorem}), to show that they are equal as submodules of $T_l(A)^{I_F}$ it is enough to check $\cF\subseteq T_l(A)^t$.
When $K$ is a local field (as is the case in Theorem \ref{main}) the eigenvalues of the  Frobenius element of $\textup{Gal}(F^{\textup{nr}}/F)$ on $\cF$ have absolute value $|k_F|$, since it acts on $H^1(\Upsilon,\mathbb{Z})$ with finite order, and on $\mathbb{Z}_l(1)$ as multiplication by $|k_F|$ (here $F^{\textup{nr}}$ denotes the maximal unramified extension of $F$).  By examining the  graded pieces of the filtration in \eqref{eq1} we see that $T_l(A)^t$ can be characterised as the largest submodule of $T_l(A)^{I_F}$ on which Frobenius acts with all eigenvalues having weight  $|k_F|$ and so the claim follows.
For general $K$ one can use Deligne's Frobenius weights argument in  \cite[I,\,\S6]{SGA7I} to reduce to this case.
\end{remark}

\begin{corollary} 
\label{corGsplit}
The canonical filtration $0\subset T_l(A)^t \subset T_l(A)^{I_F} \subset T_l(A)$ in \eqref{eq1} is $G$-stable and its graded pieces are, as $G$-modules,
$$
  H^1(\Upsilon,\Z) \tensor_\Z\Z_l(1), \qquad T_l \Pic^0 (\tilde \cC_{\bar k_F}), 
    \qquad  H_1(\Upsilon,\Z) \tensor_\Z\Z_l.
$$

\end{corollary}

\begin{proof}

It follows from Lemma \ref{limit of ST cor}, Theorem \ref{dual graph theorem} and Remark \ref{rem:toric part} that the filtration is $G$-stable and that the first two graded pieces are as claimed. Now by Grothendieck's orthogonality theorem \cite[Theorem 2.4]{SGA7I}, $T_l(A)^{I_F}$ is the orthogonal 
complement of $T_l(A)^t$ under the Weil pairing
\begin{equation*}\label{weil pairing}
  T_l(A)\times T_l(A)\rightarrow \Z_l(1)
\end{equation*}
 (here we use the canonical principal polarisation to identify $A$ with its dual).  Since the Weil pairing is $G$-equivariant this identifies the quotient $T_l(A)/T_l(A)^{I_F}$ with \[\textup{Hom}_{\Z_l}(H^1(\Upsilon,\Z) \tensor_\Z\Z_l(1),\Z_l(1))=H_1(\Upsilon,\Z) \tensor_\Z\Z_l\]
which completes the proof. 
\end{proof}

\begin{remark}[Proof of Theorem \ref{main}] 

That the filtration \eqref{eq1} is independent of $F$ follows from its characterisation in terms of the identity component of the N\'eron model in  \cite[IX, \S12]{SGA7I}, combined with the fact that the identity component of  the N\'eron model of a semistable abelian variety commutes with base change. (Alternatively, this can also be seen by considering Frobenius eigenvalues on the graded pieces.)
To deduce our main theorem, we take $C/F$ the base change of a (positive genus) curve over $K$ which becomes semistable over $F$, and take $G=G_K$ acting as in Example \ref{galois example1} (cf. also Remark \ref{assumptionshold}) throughout this section: this gives the claimed description of the graded pieces and the Tate module decomposition. 
The explicit formula for the action on non-singular points of $\cC_{k_F}(\bar{k}_F)$ follows from Theorem \ref{geommain}(3).

\end{remark}

\end{document}